\documentclass[11pt]{amsart}

\usepackage{graphicx}
\usepackage{amsmath}
\usepackage{amssymb}
\usepackage{mathrsfs} 
\usepackage{xcolor}

\usepackage[utf8]{inputenc}

\newtheorem{thm}{Theorem}

\newtheorem{lem}[thm]{Lemma}
\newtheorem{cor}[thm]{Corollary}

\theoremstyle{definition}
\newtheorem{definition}[thm]{Definition}

\newcommand{\scalprod}[2]{\langle{#1},{#2}\rangle}

\theoremstyle{remark}

\newcommand{\Set}[2]{\Big\{{#1}\,\Big|\;{#2}\Big\}}
\newcommand{\SSet}[2]{\Bigg\{{#1}\; \Bigg|\ {#2}\Bigg\}}

\newcommand{\R}{\mathbb{R}}  
\newcommand{\C}{\mathbb{C}} 
\newcommand{\Z}{\mathbb{Z}}

\newcommand{\N}{\mathbb{N}}
\newcommand{\T}{\mathbb{T}}

\newcommand{\TT}{\mathcal{T}}
\newcommand{\FF}{\mathcal{F}}
\newcommand{\E}{\mathbb{E}}
\newcommand{\A}{\mathbb{A}}

\newcommand{\AB}{\mathcal{A}}

\newcommand{\WW}{\mathfrak{W}}
\DeclareMathOperator{\M}{M}

\DeclareMathOperator{\Diff}{Diff}

\renewcommand{\epsilon}{\varepsilon}
\newcommand{\eps}{\epsilon}

\newcommand{\mres}{%
  \,\raisebox{-.127ex}{\reflectbox{\rotatebox[origin=br]{-90}{$\lnot$}}}\,%
 }

\begin{document}

\title[Alberti's type rank one theorem for martingales
]{Alberti's type rank one theorem for martingales
}

\subjclass[2020]{}

\author{Rami Ayoush}
\address{Institute of Mathematics, University of Warsaw, Banacha 2, 02-097, Warsaw, Poland}

\email{r.ayoush@uw.edu.pl}

\author{Dmitriy Stolyarov}
\address{Department of Mathematics and Computer Science, St. Petersburg State University; 199178,14th line 29, Vasilyevsky Island, St. Petersburg, Russia}

\address{St. Petersburg Department of Steklov Mathematical Institute, 191023, 27, Fontanka, St. Petersburg, Russia}
\email{d.m.stolyarov@spbu.ru}

\author{Michał Wojciechowski}
\address{Institute of Mathematics, Polish Academy of Sciences \\ Śniadeckich 8 \\00-656 Warsaw, Poland}
\email{miwoj@impan.pl}
\thanks{R.A. and M.W. were supported by the National Science Centre, Poland, CEUS programme, project no. 2020/02/Y/ST1/00072. D. S. is supported by the Basis Foundation grant no. 21-7-2-12-1.}

\begin{abstract}
We prove that the polar decomposition of the singular part of a vector measure depends on its conditional expectations computed with respect to the $q$-regular filtration. This dependency is governed by a martingale analog of the so-called wave cone, which naturally corresponds to the result of De Philippis and Rindler about fine properties of PDE-constrained vector measures. As a corollary we obtain a martingale version of Alberti's rank-one theorem.
\end{abstract}

\maketitle

%%%%%%%%%%%%%%%%%%%%%%%%%%%%%%%%%%%%%%%%%%%%%%%%%%%%%%%%%%

The main goal of this paper is to deliver yet another example of deep correspondence between Fourier analysis and martingale theory. 
%However, unlike in most cases involving this well known connection, we do not present any application of probability to analysis. Conversely, we try to find the most suitable counterpart of an existing 'real world' result in an abstract martingale setting. 
The theorem in which we are interested is an analog of the result of De Philippis and Rindler concerning polar decomposition of PDE-constrained measures (\cite{DR}). The original theorem says the following:
\begin{thm}[\cite{DR}, Theorem 1.1]\label{polar}
Let 
\[
\mathcal{A}(D) = \sum_{|\alpha| \leq r}A_{\alpha}(\partial^{\alpha}), \quad A_{\alpha} \in M_{n\times m}(\R),
\]
be a constant-coefficient linear operator that maps~$\mathbb{R}^m$-valued functions in~$N$ variables to~$\R^n$-valued functions, with the principal symbol
\[
\A^{r}[\xi] = \sum_{|\alpha| = r}A_{\alpha}\xi^{\alpha}, \quad  \xi \in \R^N.
\]
Suppose that a locally finite vector measure $\mu \in \M(\Omega; \R^{m})$, where~$\Omega \subset \R^{N}$ is an arbitrary domain, satisfies
\begin{equation}\label{DifferentialConstraint}
\mathcal{A}(D) \mu = 0 \quad \text{in}\ \mathcal{D}'(\Omega; \R^{n}).
\end{equation}
Then,
\begin{equation}\label{wavecone}
\frac{d\mu}{d|\mu|}(x) \in \bigcup_{\xi \in \R^{n}\setminus \{0\}} \ker \A^{r}[\xi] \quad \text{for} \ |\mu_{s}|\text{-a.e.} \ x.
\end{equation}
\end{thm}
In the equation \eqref{wavecone} the symbol $\frac{d\mu}{d|\mu|}(x)$ denotes the polar decomposition of $\mu$ at $x$, i.e. the value of the Radon-Nikodym derivative of $\mu$ with respect to its total variation at $x$. By~$\mu_s$ we denote the singular part of~$\mu$. The set on the right hand side of \eqref{wavecone} is called the wave cone of $\AB$. Defining a similar object associated with a filtration is perhaps the most rewarding outcome of our considerations. 
We will also derive a discrete variant of Alberti's famous rank-one theorem (see \cite{A}). 

We obtain our results on a specific metric space corresponding to the setting of $q$-regular martingales considered by Janson in \cite{Janson} for the purpose of modeling real Hardy spaces. His martingale model can be realized  on the probability space $\Omega= (\T, \sigma(\bigcup_{n=0}^{\infty} \FF_{n}), \mu)$, which we describe below.

Let $q\geq 3$ be a fixed integer. The set $\T$ will consist of all infinite paths of the infinite $q$-regular tree $\TT$ that begin from its root. To be more precise, the oriented graph $\TT = (V,E)$ is defined by the following properties:
\begin{enumerate}
\item it has a distinguished vertex called the root; \\
\item it is an infinite directed and connected graph without cycles;\\
\item each vertex has~$q$ outgoing edges; \\
\item each vertex except the root has one incoming edge, the root has no incoming edges.
\end{enumerate} 
If $v\to w$ in~$\TT$, then we refer to $w$ as a son of $v$ and write $w^{\uparrow} = v$. It will be convenient to enumerate the sons of a vertex~$w$ with the numbers~$1,2,\ldots, q$ and fix such an enumeration. 

By a path we mean an infinite directed sequence of vertices starting from the root, each succeeding being a son of the preceding. If~$x \in \T$ is a path, then we denote the~$n$th element of the corresponding sequence by~$x(n)$. 
The set of atoms (of $n$th generation) $\AB \FF_{n}$ consists of $q^{n}$ sets of the form
\[
\omega_{v} = \{x\in \T: x(n) = v \} \ \text{for} \ v \ \text{such that} \ d_{\TT}(root, v) = n.
\]
In other words, for a vertex $v$ whose standard graph distance to the root is~$n$, the set $\omega_{v}$ consists of infinite paths that pass through~$v$. The collection of the sets~$\AB\FF_n$ forms a partition of~$\T$ and we put $\FF_{n}$ to be the set algebra generated by~$\AB \FF_{n}$. Further, we tacitly transfer the tree structure from $V$ to the set of all atoms of all generations, writing $\omega_{v}^{\uparrow} = \omega_{w}$ if $v^{\uparrow} = w$, etc. and we will not make any distinction between vertices of a tree and atoms. Finally, the measure $\mu$ is simply the uniform measure on $\T$, i.e. $\mu(\omega) = q^{-n}$ for $\omega \in \AB \FF_{n}$.

The space~$\T$ played an important role in modeling so-called Bourgain--Brezis inequalities, see~\cite{ASW} and~\cite{St22}. See the first of these papers for more information about~$\T$.

A sequence of~$\R^l$-valued functions~$\{F_n\}_n$ is a martingale provided for any~$n\in \mathbb{N}\cup \{0\}$ the function~$F_n$ is~$\FF_n$-measurable and
\begin{equation}\label{MartingaleProperty}
\E(F_{n+1}\mid \FF_n) = F_n.
\end{equation}
 %If~$\{F_n\}_n$ is a martingale, then the functions~$dF_n = F_n-F_{n-1}$ are called martingale differences.
%A natural realization of this structure is given by the set of $q$-adic subintervals of $[0,1)$ and without doing a big harm to generality of results the Reader may assume that $\AB \FF_{n}$ consists of all $q$-adic intervals of size $q^{-n}$. In this case the tree structure is given by the relation of inclusion and $\mu$ is just one-dimensional Lebesgue measure.
Each finite~$\R^l$-valued measure~$\nu \in \M(\T, \R^{l})$ on~$\T$ generates a martingale~$\{\nu_n\}_n$ by the formula 
\[
\nu_{n}(x) = q^{n} \nu(\omega),\qquad x\in \omega \in \AB\FF_n.
\] 
\begin{definition}\label{DOmega}
Let us denote by $\Diff(\nu)$ the set of all matrices $D_{\omega} \in M_{q\times l}(\R)$ of the form
\begin{equation}\label{MartingaleDifferenceMatrix}
D_{\omega} = \begin{bmatrix}
	\ | & \dots & |  \\
	d_{1} & \dots & d_{q} \\
	| & \dots & |
\end{bmatrix},
\end{equation}
where $\omega \in \AB \FF_{n}$ and 
\[
d_{i} = \nu_{n+1}(\omega_{i}) - \nu_n(\omega),
\]
and $\omega_{1}, \dots, \omega_{q} \in \AB \FF_{n+1}$ are all sons of $\omega$.
\end{definition}
\begin{definition}
Let~$W\subset M_{q\times l}(\R)$ be a linear subspace. We denote the space~$\WW$ by the rule
\begin{equation}
\WW = \{\nu\in \M(\T,\R^l)\colon \Diff(\nu) \subset W\}.
\end{equation}
\end{definition}
Note that by the martingale property~\eqref{MartingaleProperty} each row of the matrix~\eqref{MartingaleDifferenceMatrix} is a vector with zero mean. Therefore, we may restrict our attention to subspaces of the form~$W\subset \R_0^q\otimes \R^l$ only. In~\cite{ASW}, the space~$\WW$ was called a martingale Sobolev space. The terminology 'martingale~$\mathrm{BV}$-type space' seems more appropriate. The space~$W$ is an analog of the differential constraint~\eqref{DifferentialConstraint} in the sense that the spaces~$\WW$ and~$\{\mu \in \M(\R^N,\R^m)\mid \mathcal{A}(D)\mu = 0\}$ have many similarities (say, they behave similarly under the action of Riesz potentials, see~\cite{ASW} and~\cite{St22}; another confirmation of this principle comes from dimensional estimates for corresponding measures, see~\cite{ASW},~\cite{Sto2023}, and~\cite{A2021}).
\begin{definition}
Let $v\in \R^{l}$ and $A\subset M_{q\times l}(\R)$ be a subset of real $q\times l$ matrices. We define the rank-one angle between $v$ and $A$ as
\[
\gamma(v,A):=\inf\{|\angle(v\otimes w, m)|_{HS}: w\in \R^{q}_0, w\neq 0, m\in A \},
\]
where
\[
|\angle(A_{1}, A_{2})|_{HS} = \arccos \bigg{(}\frac{|tr(A_{1}^{t} A_{2})|}{\lVert A_{1}\rVert_{HS}\lVert A_{2}\rVert_{HS}}\bigg{)},\quad A_1,A_2\in M_{q\times l}.
\]
is the measure of the angle between matrices, computed with respect to the Hilbert--Schmidt norm~$\|\cdot\|_{HS}$. Here, $v\otimes w = v\cdot w^{t}$.
\end{definition}

\begin{definition}
We define the martingale wave cone of $W$ as the set
\begin{equation*}
\Lambda(W) = \{v\in \R^l\colon \exists w\in \R^q_0\setminus \{0\} \ \text{such that}\ v\otimes w \in W\}.
\end{equation*}
In other words,
\[
\Lambda(W) = \{v\in \R^{l}: \gamma(v, W) = 0\}.
\]

\end{definition} 
Let us decompose $\nu$ into absolutely continuous and singular part (with respect to the uniform measure on $\T$)
\[
\nu = \nu_{abs}+\nu_{s}.
\]
Our main result is the following:
\begin{thm}\label{mainres}
Let $\nu \in \WW$ be a finite $\R^{l}$-valued measure. Then
\[
\frac{d\nu}{d|\nu|}(x) \in \Lambda(W) \quad \text{for} \ |\nu_{s}|\text{-a.e.} \ x.
\]
\end{thm}

\section{Decomposition into flat and convex atoms}
The proof of Theorem ~\ref{mainres} relies on a combination of the ideas from \cite{Janson} and \cite{ASW}. The crucial tool is the decomposition of $\T$ into parts corresponding to the so-called $\epsilon$-convex and $\epsilon$-flat atoms introduced in \cite{ASW}.
\begin{definition}
For a given $\epsilon \in (0,1)$, an atom $\omega\in \AB\FF_{n}$ is called $\epsilon$-convex if 
\[
\E(\lVert \nu_{n+1} \rVert - \lVert \nu_{n} \rVert) 1_{\omega} \geq \epsilon \E\lVert\nu_{n}\rVert1_{\omega}.
\]
If the reverse inequality holds, then $\omega$ is an~$\epsilon$-flat.
\end{definition}
Here and in what follows, we use the standard Euclidean norm on~$\R^l$. One may see that~$0$-flat atoms correspond to the case where the matrix~$D_w$ has rank one, because in this case the triangle inequality~$\E\|\nu_{n+1}\|1_{\omega} \geq \|\nu_n(\omega)\|$ turns into equality.

Let us denote by $\TT^{\epsilon}$ the subgraph of $\TT$ generated by the vertices corresponding to $\epsilon$-flat atoms. One can represent $\TT^{\epsilon} = \cup \TT^{\epsilon}_{i}$, where each $\TT^{\epsilon}_{i}$ is a maximal by inclusion connected subgraph (a tree). It turns out that the singular part of $\nu$ is  carried by infinite paths of the trees $\TT^{\epsilon}_{i}$.

\begin{definition}
	We call a point $x \in \T$  an $\epsilon$-leaf if there exists $n_{0}\in \N$  and a sequence of $\epsilon$-flat atoms~$\{\omega_n\}_n$, $\omega_{n} \in \AB \FF_{n}$ such that $x \in \omega_{n}$ for all $n\geq n_{0}$. Let us denote  by $L(\epsilon)$ the set of all $\epsilon$-leaves and put
	\[
	L := \bigcap_{\epsilon > 0} L(\epsilon).
	\]
	For a subgraph $G \subset \T$, we denote $L(\epsilon, G)$ the set of all infinite paths of $G$ that are restrictions of $\epsilon$-leaves to $G$.
\end{definition}

By Lemma 3.3 and Corollary 3.4 from \cite{ASW}, we have~$|\nu_{s}|(\T \setminus L(\epsilon)) = 0$ for all~$\eps > 0$
and so 
\begin{equation}\label{leavessupp}
	|\nu_{s}|(\T\setminus L) = 0.
\end{equation}

We will need yet another classification of leaves.
\begin{definition}\label{bigleaves}
We call $x\in L$ a big leaf if there exists $\beta >0$ and a sequence of atoms
\[
\omega_{n_{1}} \supset \omega_{n_{2}} \supset \dots \supset \{x\}, \quad \ \ n_{1} < n_{2} < \dots
\]
such that $\omega_{n_{k}} \in \AB \FF_{n_{k}}$ and 
\begin{equation}\label{bigleaf}
	\frac{1}{q} \sum_{j=1}^{q} \lVert d_{j}^{(n_{k})} \rVert \geq \beta \lVert \nu_{n_{k}}(\omega_{n_{k}}) \rVert,
\end{equation}
where 
\[
D_{\omega_{k}} = \begin{bmatrix}
	\ | & \dots & |  \\
	d_{1}^{(n_{k})} & \dots & d_{q}^{(n_{k})} \\
	| & \dots & |
\end{bmatrix}.
\]
Otherwise, we call $x\in L$ a small leaf. The sets of small and big leaves will be denoted by $L_{s}$ and $L_{b}$, respectively.
\end{definition}

%\begin{remark}
%For example, points at which  $|\nu|$ is $\alpha$-H\"older with $\alpha<1$ satisfy \eqref{bigleaf}, thus intuitively \eqref{bigleaf} is satisfied on $\alpha$-dimensional part of $\nu$. For such points, as we will see, Theorem ~\ref{mainres} is somehow easier to prove and follows from purely geometrical arguments. A similar effect happens also in Theorem ~\ref{polar} which is much easier to prove on the rectifiable part of a measure.
%\end{remark}

\section{Proof of the main theorem}
We need two algebraic lemmas. The first one quantifies the `flattening effect' (c.f. Lemma 2.1. in \cite{ASW}). The notation~$\pi_a x$ means projection of~$x\in \R^l$ onto the line spanned by~$a\in \R^l$,~$\pi_{a^\perp} x$ denotes the projection onto the orthogonal complement of~$a$.
\begin{lem}\label{lem1}
Let $\omega \in \AB \FF_{n}$ be an $\epsilon$-flat atom with~$\eps < 1$, i.e.
\[
\frac{1}{q} \sum_{j=1}^{q} \lVert a+d_{j} \rVert - \lVert a \rVert \leq \epsilon \lVert a \rVert
\]
for 
\[
D_{\omega} = \begin{bmatrix}
	\ | & \dots & | \\
	d_{1} & \dots & d_{q} \\
	| & \dots & |
\end{bmatrix}, \quad a=\nu_{n}(\omega).
\]
Then, for all $j=1, 2,\ldots, q$ we have $\lVert \pi_{a^{\perp}}d_{j} \rVert \leq 2q\sqrt{\epsilon} \lVert a \rVert$.
\end{lem}

\begin{proof}
From the triangle inequality and~$\sum_j d_j=0$, we have
\begin{equation}
\frac{1}{q} \sum_{j=1}^{q} \lVert a+\pi_a d_{j} \rVert \geq \lVert a \rVert.
\end{equation}
The above and the definition of $\epsilon$-flat atom imply that
\begin{multline}\label{orthogonal}
\epsilon \lVert a \rVert \geq \frac{1}{q} \Big(\sum_{j=1}^{q} \lVert a+d_{j} \rVert - \lVert a + \pi_{a} d_{j} \rVert\Big) = \frac{1}{q} \sum_{j=1}^{q} \frac{\lVert \pi_{a^{\perp}}(d_{j})\rVert^{2}}{\lVert a+d_{j} \rVert + \lVert a + \pi_{a} d_{j} \rVert} \\ \geq \frac{1}{2q} \sum_{j=1}^{q} \frac{\lVert \pi_{a^{\perp}}(d_{j})\rVert^{2}}{\lVert a+d_{j} \rVert} \geq \frac{1}{4q^{2}} \sum_{j=1}^{q} \frac{\lVert \pi_{a^{\perp}}(d_{j})\rVert^{2}}{\lVert a \rVert}.
\end{multline}
The latter inequality follows from 
\begin{equation}
\lVert a+d_{j} \rVert \leq \sum_{k=1}^{q} \lVert a+d_{k} \rVert \leq (1+\epsilon)q\lVert a \rVert \leq 2q\lVert a \rVert.
\end{equation}
Thus,~\eqref{orthogonal} yields $\lVert \pi_{a^{\perp}} d_{j} \rVert^{2} \leq 4q^{2} \epsilon \lVert a \rVert^{2}$  for all $j=1, 2 \ldots, q$.
\end{proof}

The second lemma uses the smoothness of the Euclidean norm (c.f. Lemma 10 in \cite{Janson}).
\begin{lem}\label{lem2} Suppose that $\omega \in \AB \FF_{n}$ and
\[
D_{\omega} = \begin{bmatrix}
	\ | & \dots & |  \\
	d_{1} & \dots & d_{q} \\
	| & \dots & |
\end{bmatrix}, \quad a=\nu_{n}(\omega).
\]
Let us assume that~$\gamma(a, \{D_{\omega}\}) \geq \eta >0$
and~$\sum_{j=1}^{q} \lVert d_{j} \rVert \leq \delta \lVert a \rVert$
for some parameters $\eta, \delta>0$. Then, for sufficiently small $\delta$, there exists $p_{0} = p_{0}(\delta, \eta)$ such that
\begin{equation}\label{eq14}
	\lVert a \rVert^{p} \leq \frac{1}{q} \sum_{j=1}^{q} \lVert a+d_{j} \rVert^{p}
\end{equation}
for any $p$ satisfying  $p_{0}<p<1$.
\end{lem}
\begin{proof}
Without loss of generality, we may assume~$\|a\|=1$. By duality,
\begin{multline}
	\frac{\sqrt{\sum_{j=1}^{q} |\langle a, d_{j} \rangle|^{2}}}{\sqrt{\sum_{j=1}^{q}\lVert d_{j} \rVert^{2}}} = \sup_{\lVert \{\epsilon_{j}\} \rVert = 1} \frac{\sum_{j=1}^{q} \langle \epsilon_{j} a, d_{j} \rangle}{\sqrt{\sum_{j=1}^{q}\lVert d_{j} \rVert^{2}}} \\ 
	= \sup_{\lVert \{\epsilon_{j}\} \rVert = 1} \cos|\angle ( a \otimes \{\epsilon_{j}\}, D_{\omega})|_{HS} \leq \cos \eta,
\end{multline}
which leads to
\begin{equation}\label{eq16}
\sum_{j=1}^{q} |\langle a, d_{j} \rangle|^{2} \leq \cos^2 \eta  \sum_{j=1}^{q}\lVert d_{j} \rVert^{2}.
\end{equation}
Using the representation
\begin{equation}
\|a+d_j\|^p = \Big(1+ 2\scalprod{a}{d_j} + \|d_j\|^2\Big)^{p/2},
\end{equation}
and treating the~$d_j$ as small parameters, we apply Taylor's formula to the right hand side of~\eqref{eq14}:
	\begin{multline}\label{taylor}
		\sum_{j=1}^{q} \lVert a+d_{j} \rVert^{p} =  \sum_{j=1}^{q}\bigg{(}1+\langle a, d_{j} \rangle + \frac{p}{2} \lVert d_{j} \rVert^{2}+ \frac{p(p-2)}{2}\scalprod{a}{d_j}^2\bigg{)} \\ +O\bigg{(}\sum_{j=1}^{q} \lVert d_{j} \rVert^{3}\bigg{)} =
		 q+\frac{p}{2}\sum\limits_{j=1}^q \lVert d_{j} \rVert^{2} + \frac{p(p-2)}{2}\sum\limits_{j=1}^q\scalprod{a}{d_j}^2 + O\bigg{(}\sum_{j=1}^{q} \lVert d_{j} \rVert^{3}\bigg{)}.
	\end{multline}
Using~\eqref{eq16}, we bound the right hand side of \eqref{taylor} from below by
\begin{equation}
q + \frac{p}{2} \sum_{j=1}^{q} \lVert d_{j} \rVert^{2}+ \cos^2\eta\frac{p(p-2)}{2}\sum_{j=1}^{q}\|d_j\|^2+O\bigg{(}\sum_{j=1}^{q} \lVert d_{j} \rVert^{3}\bigg{)}.
\end{equation}
Since $|\cos \eta| < 1$, the last expression is at least $q$ provided that $\delta$ is sufficiently small and $p$ is sufficiently close to one. This justifies the desired inequality.
\end{proof}

\begin{proof}[Proof of Theorem \ref{mainres}]
	By the Besicovitch--Lebesgue differentiation theorem\footnote{We are applying a differentiation theorem on a special metric space; see clarification at the beginning of Subsection~$4.2$ in~\cite{ASW}.} we have that for $|\nu|$-a.e $x$
\begin{equation}\label{lebdif}
	\lim_{n \to \infty} \angle\Big(\nu_{n}(x),\frac{d\nu}{d|\nu|}(x)\Big) = 0.
\end{equation}
In particular, this is true for $|\nu_{s}|$-a.e. $x\in L$. For the sake of presentation, let us assume that this is true for all points from $L$.
By \eqref{leavessupp}, it suffices to disprove that there exists $\eta>0$ such that
\begin{equation}
	B_{\eta} = \bigg{\{}x\in L: \gamma\bigg{(}\frac{d\nu}{d|\nu|}(x), \Diff(\nu)\bigg{)} > \eta \bigg{\}}
\end{equation}
has positive $|\nu_{s}|$-measure, or equivalently by \eqref{lebdif} to disprove that
\begin{equation}
	\exists\ n_{0}\quad \forall \ n \geq n_{0}\qquad |\nu_{s}|(B_{\eta,n})>0,
\end{equation}
where
\begin{equation*}
	B_{\eta,n} = \bigg{\{}x\in L: \gamma\bigg{(}\nu_{n}(x), \Diff(\nu)\bigg{)}>\frac{\eta}{2}\bigg{\}}.
\end{equation*}
Let us fix $n$ and decompose $B_{\eta,n} = B_{1}\cup B_{2}$ into sets consisting of big and small leaves, respectively.

\textit{Step 1.} $|\nu|(B_{1}) = 0$. Let $x\in B_{1}$ be a big leaf and $\{\omega_{n_{k}}\}, \{D_{\omega_{n_{k}}}\}$ and $\beta$ be as in Definition \ref{bigleaves}. Put $a = \nu_{n_k} (\omega_{n_{k}})$. We will show that $\gamma(a, \Diff(\nu))$ is in fact arbitrarily small for sufficiently large~$k$. Let us assume that $\lVert a \rVert = 1$. In such a case,~\eqref{bigleaf} and the Cauchy--Schwarz inequality yield
\begin{equation*}
\sum_j\|d_j^{(n_k)}\|^2 \geq \frac{\beta^2}{q},
\end{equation*}
and Lemma~\ref{lem1} implies
\begin{equation*}
\sum_j \|\pi_{a^\perp}(d_j^{(n_k)})\|^2 \leq (2q^2\sqrt \eps)^2 = 4q^4\eps.
\end{equation*}

We have
\begin{multline}\label{eq20}
	\sup_{v\in \R^{q}\setminus \{0\}} \frac{tr[(v\otimes a)^{t}D_{\omega_{n_{k}}}]}{\lVert v\otimes a \rVert_{HS}\lVert D_{\omega_{n_{k}}} \rVert_{HS}} = \sup_{v\in \R^{q}\setminus \{0\}} \frac{\sum_{j} v_{j}\langle a, d_{j}^{(n_{k})} \rangle}{\lVert v \rVert \lVert a \rVert \sqrt{\sum_{j} \lVert d_{j}^{(n_{k})} \rVert^{2}}} = \\ \bigg{(} \frac{\sum_{j} |\langle a, d_{j}^{(n_{k})} \rangle|^{2}}{\sum_{j} \lVert d_{j}^{(n_{k})}\rVert^{2}} \bigg{)}^{1/2} = \bigg{(} \frac{\sum_{j} \lVert d_{j}^{(n_{k})} \rVert^{2} - \sum_{j}\lVert \pi_{a^{\perp}} (d_{j}^{(n_{k})}) \rVert^{2}}{\sum_{j} \lVert d_{j}^{(n_{k})}\rVert^{2}} \bigg{)}^{1/2} \\
	\geq \bigg{(} 1 - \frac{4q^5\eps}{\beta^2}\bigg{)}^{1/2}.
\end{multline}
Now it suffices to notice that for sufficiently large $k$, $\omega_{n_{k}}$ is $\epsilon$-flat for arbitrarily small $\epsilon$. Thus, we have $\gamma(a, \Diff(\nu)) < \frac{\eta}{2}$ from~\eqref{eq20}. Consequently,~$\gamma(\frac{d\nu}{d|\nu|}(x), \Diff(\nu)) < \eta$ and $B_{1} = \emptyset$.

\textit{Step 2.} $|\nu|(B_{2}) = 0$. Assume the contrary. Then, there exists $\epsilon>0$ such that $|\nu_{s}|(L(\epsilon) \cap B_{2})>0$. Consider the decomposition $\TT^{(\epsilon)} = \cup_{j} \TT^{(\epsilon)}_{j}$. One can find $j$ such that $L(\epsilon, \TT^{(\epsilon)}_{j})$ has positive $|\nu_{s}|$-measure (by the disjointedness of those sets). Now it is time to use Lemma \ref{lem2}. Assume first that the inequality reverse to \eqref{bigleaf} holds for all $\omega \in \TT_{j}^{(\epsilon)}$ with a suitable small $\delta$ required in this lemma. Then, on the one hand we have the property that
\begin{equation}\label{case2}
	\theta = \nu^{\TT_{j}^{(\epsilon)}}{\mres L(\epsilon, \TT_{j}^{(\epsilon)})} \text{ and } \nu{\mres L(\epsilon, \TT_{j}^{(\epsilon)})} \text{ have the same singular part.}
\end{equation} 
Here~$\nu^{\TT_{j}^{(\epsilon)}}$ denotes the limit measure of the martingale whose evolution is restricted to the tree $\TT_{j}^{(\epsilon)}$ (if we leave the tree, then we stop the martingale).
On the other hand, for $p<1$ given by Lemma \ref{lem2}, the sequence $\lVert \E(\theta | \FF_{n}) \rVert^{p}$ is a positive submartingale, which, by Doob's theorem on the boundedness of the martingale maximal function in~$L_q$ with~$q > 1$, implies that the maximal function of~$\theta$ is summable, and $\theta$ lies in the martingale space $H^{1}(\R^l)$ (for the details see p. 148 in \cite{Janson}).  Thus, $\theta$ is absolutely continuous.
	
	If the inequality \eqref{bigleaf} is not satisfied for all $\omega \in \TT_{j}^{(\epsilon)}$, then we use the fact that for each infinite path it must be true for atoms that are sufficiently far from the root, i.e.
\begin{equation}
	\forall{x \in L(\epsilon, \TT_{j}^{(\epsilon)})} \ \exists N 
	\ \forall n\geq N \text{ the inequality }\eqref{bigleaf} \text{ holds for } \omega\in\AB\FF_{n} \text{ if } x\in \omega.
\end{equation}
From this we can cover $L(\epsilon,\TT_{j}^{(\epsilon)})$ by a countable union of disjoint sets of the form $L(\epsilon, \TT'_{k})$ where $\TT'_{k}$ are some trees and one of them gives a rise to a measure satisfying \eqref{case2} and whose leaves form a set of positive $|\nu_s|$-measure, leading to a contradiction.
\end{proof}

\section{Martingale rank-one theorem}
In this section we will present an analog of famous Alberti's rank-one theorem. For simplicity, we will formulate and prove only the two-dimensional special case. The extension to higher dimensions is straightforward.

We consider a specific space~$W$. We assume~$q=m^2$ for some~$m\in \N$ and identify the set~$1,2,\ldots,q$ with the group~$(\Z/m\Z)^2$; here~$\Z/m\Z$ is the group of residues modulo~$m$. Then, the elements of~$M_{q\times l}$ are naturally identified with~$\R^l$ valued functions on the `discrete torus'~$(\Z/m\Z)^2$. We set~$l=8$ and also identify~$\R^l$ with the space of~$2\times 2$ complex matrices. With this notation, define the space~$W$ by the formula
\begin{equation}\label{SpecBVspace}
W = \SSet{D\in \R_0^{m^2}\times \R^4}{\begin{aligned}&\exists f,g\colon (\Z/m\Z)^2\to \C \quad \forall i,j =1,2,\ldots,m\\ &D_{i,j} = \begin{pmatrix} f(i+1,j) - f(i,j) & f(i,j+1) - f(i,j)\\
g(i+1,j)-g(i,j)& g(i,j+1) - g(i,j)
\end{pmatrix}
\end{aligned}
}.
\end{equation} 
The space~$\WW$ generated by this~$W$ somehow resembles the space of~$\mathrm{BV}$ maps. In particular, the corollary below may be thought of as a martingale version of Alberti's theorem from~\cite{A}.
\begin{cor}
Let~$W$ be given by~\eqref{SpecBVspace}, let~$\nu \in \WW$. Them,~$\frac{d\nu}{d|\nu|}$ is a matrix of rank one for~$|\nu_s|$ almost all~$x$.
\end{cor}
\begin{proof}
To derive the corollary from Theorem~\ref{mainres}, we need to show that any matrix in the martingale wave cone~$ \Lambda(W)$ has rank one. We will describe the cone~$\Lambda(W)$ using the Fourier transform on~$(\Z/m\Z)^2$ (see Section~$7$ in~\cite{Sto2023b} for a more detailed exposition of similar material). We may describe~$W$ as
\begin{multline}\label{eq22}
\Set{D}{\forall \gamma \in (\Z/m\Z)^2 \setminus \{0\}\quad \hat{D}(\gamma) \in \Omega(\gamma)},\ \text{where}\ \Omega(\gamma) =\\
\mathrm{span}\bigg(\begin{pmatrix}e^{2\pi i \gamma_1/m}-1& e^{2\pi i \gamma_2/m}-1\\ 0&0\end{pmatrix},\begin{pmatrix}0&0\\ e^{2\pi i \gamma_1/m}-1& e^{2\pi i \gamma_2/m}-1\end{pmatrix}\bigg).
\end{multline}
Pick some~$v\in \Lambda(W)$, let~$v\otimes w \in W$, where~$w\in \R_0^{m^2}\setminus \{0\}$. Then, with the notation~$\mathcal{F}$ for the Fourier transform,
\begin{equation*}
\mathcal{F}[v\otimes w](\gamma) = v\otimes \hat{w}(\gamma).
\end{equation*}
Since~$w$ is not a constant function,~$\hat{w}(\gamma)\ne 0$ from some~$\gamma \in (\Z/m\Z)^2\setminus \{0\}$. Then, by~\eqref{eq22}, 
\begin{equation*}
v = \Big(e^{2\pi i \gamma_1/m}-1,\; e^{2\pi i \gamma_2/m}-1\Big)\otimes (a,b),
\end{equation*}
where~$(a,b)\in \C^2$ is a non-zero vector.
\end{proof}
\bibliography{biblio} 
\bibliographystyle{alpha}

\end{document}